\documentclass[12pt,psfig,reqno]{amsart}
\usepackage{amssymb,amsfonts,latexsym}
\usepackage{graphics,verbatim}
\usepackage{graphicx}
\usepackage[usenames]{color} 
\usepackage{soul} 

\hoffset=-1cm \errorcontextlines=0 \numberwithin{equation}{section}

 \theoremstyle{plain}
\newtheorem{theorem}{Theorem}[section]
\newtheorem{lemma}[theorem]{\bf{Lemma}}

\newenvironment{pf}{{\noindent \bf Proof.\/}}{\hfill$\Box$}

\date {}

\begin{document}

\title[Harnack's inequality and Green functions]{Harnack's inequality and Green functions on locally finite graphs}

\author{Li Ma}

\address{Li Ma, Department of mathematics \\
Henan Normal university \\
Xinxiang, 453007 \\
China}

\email{lma@tsinghua.edu.cn}

\thanks{ The research is partially supported by the National Natural Science
Foundation of China (No.11271111).}

\begin{abstract}
In this paper we study the gradient estimate for positive solutions
of Schrodinger equations on locally finite graph. Then we derive
Harnack's inequality for positive solutions of the Schrodinger
equations. We also set up some results about Green functions of the
Laplacian equation on locally finite graph. Interesting properties
of Schrodinger equation are derived.

{ \textbf{Mathematics Subject Classification 2000}: 31C20, 31C05}

{ \textbf{Keywords}: locally finite graph, Harnack's inequality,
Green functions}
\end{abstract}

\maketitle

\section{Introduction}

Just like Poisson equation on graphs, the Schrodinger equations on
graphs are of fundamental importance. Such equations arise naturally
from the discrete process of their continuous counterparts. The
discrete equations are also useful to numerical purpose. It is a
interesting topic to understand the solution structure for linear
Schrodinger equation on graphs. In this paper we study the Harnack
inequalities for positive solutions of a class of Schrodinger
equations on locally finite graphs. We derive local Harnack
inequalities for positive solutions to Schrodinger equations based
on an improved gradient estimates in \cite{MW} (see also \cite{LY}
for more related works). We use this oppertunaty to point out that
Kato's inequalities had been previously found in \cite{DK}). Kato's
inequalities have been used by us to study the Ginzburg-Landau
equation \cite{MW} (see also \cite{M} for more background). We use
Kato's inequality and the maximum principle to understand the
principal eigenvalues of Schrodinger equations. Just as we have
known for elliptic partial differential equations of second order,
the Harnack inequality for Schrodinger equation is a basic tool to
obtain existence and compactness results for their solutions.
Harnack inequality on some special graphs is obtained in \cite{CY}.

Green's functions had been introduced in a famous essay by George
Green in 1828 and have been extensively used in solving differential
equations from mathematical physics. In particular, Green functions
are very useful to solve the Poisson equations. The concept of
Green's functions has also had a pervasive influence in numerous
areas. As pointed out ed in \cite{CY2}, Green's functions provide a
powerful tool in dealing with a wide range of combinatorial
problems. Many formulations of Green's functions occur in a variety
of topics. There is some interesting works about Green functions on
graph \cite{Ch}. Based on our Harnack inequality, we consider the
existence of global Green functions for discrete Laplace equations
defined on locally finite connected graphs. Although we use
differential tools, we can get results in graphs which have their
counterparts in Riemannian geometry.

The plan of the paper is below. Notations and Harnack inequalities
for positive solutions to schrodinger equations are introduced in
section \ref{sect2}. All results about Green functions are stated
and proved in section \ref{sect3}.

\section{gradient estimate and Harnack's inequality}\label{sect2}
We first recall some definitions and results from the book \cite{Ch}
and from our paper \cite{MW}.

Let $(X, {\mathcal E})$ be a graph with countable vertex set $X$ and
edge set ${\mathcal E}$.  We assume that the graph is {\it simple},
i.e., no loop and no multi-edges. We also assume that the graph is
connected and each vertex has finite neighbors. We simply call the
graphs with these properties \emph{the locally finite connected
graphs}. Let $\mu_{xy}= \mu_{yx}
>0$ is a symmetric weight on ${\mathcal E}$. We define $d_x =
\sum_{(x,y)\in{\mathcal E}} \mu_{xy}$
(we also assume $d_x<\infty$ for all $x\in X$) the {\it degree of } $x\in X$.\\
We use $d(x,y)$ to denote the distance between the vertices $x$ and
$y$ in $X$.

Denote by
\[
\ell(X) = \{ u:\ u:\ X \longrightarrow {\Bbb R}\},
\]
the set of all real functions (or complex-valued functions with
$\Bbb R$ replaced by $\Bbb C$ on $X$). The integral of $u$ over $X$
is defined by $\int_X u dx=\sum_{x\in X} u(x)d_x$ if the latter sum
makes sense.

We define the {\it Laplacian operator} $\Delta: \ \ell(X)
\longrightarrow \ell(X)$:
\[
(\Delta u) (x) = \sum_{(x,y)\in {\mathcal E}}
\frac{\mu_{xy}}{d_x}\big( u(y) - u(x)\big).
\]
We also define
\[
|\nabla u|^2 (x) = \sum_{(x,y)\in {\mathcal E}} \frac{\mu_{xy}}{d_x}
\big(u(y) - u(x)\big)^2.
\]
From the definition, we know that
$$
(\Delta u(x))^2\leq |\nabla u|^2 (x).
$$

Then in our previous paper \cite{MW}, we have proved the following
results.

\begin{lemma} \label{lemma-Kato}
(First Kato's inequality) For a graph $X$, we have
\[
| \nabla u|^2 \ge \big|\nabla |u| \big|^2
\]
and  (Second Kato's inequality)
\begin{eqnarray}
\Delta |u| & \ge {\rm sign} (u) \Delta u, \label{eq-kato-2}\\
\Delta u_+ & \ge {\rm sign}_+(u) \Delta u.\label{eq-kato-3}
\end{eqnarray}
Here $u_+=\frac{|u|+u}{2}$ is the positive part of the function $u$.
\end{lemma}

Recall that, by elementary computation \cite{MW}, we have
\begin{lemma}\label{lemma2}
$\Delta u^2 = 2 u \Delta u + |\nabla u|^2.$ Furthermore, if $u\Delta
u\geq 0$, then
$$
|\nabla u|^2\leq \Delta u^2.
$$
\end{lemma}

We now get the gradient estimate for positive solutions to the
(stationary) Schrodinger equation.
\begin{theorem}\label{gradient}
Assume that $u,\ Q \in \ell(X)$, $u \ge 0$, such that $-\Delta u + Q
u = 0$. Then
\[ \label{bound}
|\nabla u|^2(x) \le P(x) u^2(x), \quad \forall x\in X,
\]
where  $P(x)=\hat{d}_x(1+Q(x))^2 - 2 Q(x) -1$ with the constant
$\hat{d}_x=\sup_{(x,y)\in {\mathcal E}} \frac{d_x}{\mu_{xy}}$.
\end{theorem}

\begin{pf}
Observe that
\[
\Delta u (x) = \sum_{(x,y)\in {\mathcal E}} \frac{\mu_{xy}}{d_x}
\big( u(y) - u(x) \big) = \sum_{(x,y)\in {\mathcal E}}
\frac{\mu_{xy}}{d_x} u(y) - u(x).
\]
Hence,
\[
\sum_{(x,y)\in{\mathcal E}} \frac{\mu_{xy}}{d_x} u(y) = \Delta u(x)
+ u(x).
\]
By definition, we have
\[
\begin{aligned}
|\nabla u|^2 (x) & = \sum_{(x,y)\in{\mathcal E}} \frac{\mu_{xy}}{d_x} \big( u(y) -u(x) \big)^2 \\
& = \sum_{(x,y)\in{\mathcal E}} \frac{\mu_{xy}}{d_x} \left( -2u(x) \big( u(y) - u(x) \big) -u^2(x) + u^2(y) \right)\\
& = -2 u(x) \sum_{(x,y)\in{\mathcal E}} \frac{\mu_{xy}}{d_x} \big( u(y) - u(x) \big) -u^2(x) + \sum_{(x,y)\in{\mathcal E}} \frac{d_x}{\mu_{xy}} \left(\frac{\mu_{xy}}{d_x} u(y) \right)^2\\
& \le -2 u(x) \Delta u(x) - u^2(x) + \hat{d}_x\left(\sum_{(x,y)\in{\mathcal E}} \frac{\mu_{xy}}{d_x} u(y) \right)^2 \\
& = -(2 Q(x) + 1) u^2(x) + \hat{d}_x\big( \Delta u(x) + u(x) \big)^2\\
& = -(2 Q(x) +1) u^2(x) + \hat{d}_x \big( 1 + Q(x) \big)^2 u^2(x)\\
& = \left( \hat{d}_x(1+ Q(x) )^2 - 2Q(x) -1 \right) u^2(x).
\end{aligned}
\]
 In the first inequality,
 we have used $\frac{\mu_{xy}}{d_x} u(y) \ge 0$ for all $y\in X$ such
that $(x,y)\in{\mathcal E}$.
\end{pf}

From the result above we can easily derive the following Harnack
equality.

\begin{theorem}\label{harnack} Assume that $u,\ Q \in \ell(X)$, $u \ge 0$, such that
$-\Delta u + Q u = 0$. Given any finite set $S\subset X$, there
exists a uniform constant $C=C(S)$ such that
\begin{equation}\label{inf-sup}
\sup_S u(x)\leq C\inf_S u(x).
\end{equation}
\end{theorem}
\begin{pf}
We may assume that
$$
\inf_S u(x)=u(x_1),\ \  \sup_S u(x)=u(y_1).
$$
Then we take the minimizing path $x_1\sim x_2 \sim ...\sim x_n=y_1$
in $X$. Then $n=d(x_1,y_1)\leq diam (S)$. Recall that $|\nabla
u(x_j)|^2=\sum_{(x_j,y)\in {\mathcal E}} \frac{\mu_{x_jy}}{d_{x_j}}
\big(u(y) - u(x_j)\big)^2$. Then we have
$$
|\nabla u(x_j)|^2\geq \frac{\mu_{x_jx_{j+1}}}{d_{x_j}}
\big(u(x_{j+1}) - u(x_j)\big)^2.
$$
Hence, by ().
$$
u(x_{j+1})\leq u(x_j)+\sqrt{\frac{d_{x_j}}{\mu_{x_jx_{j+1}}}}|\nabla
u(x_j)|\leq [1+\sqrt{\hat{d}_j P(x_j)}]u(x_j).
$$
By induction, we get the bound (\ref{inf-sup}).
\end{pf}

As an easy consequence, we can get the compactness result below.

\begin{theorem}\label{compactness} Assume that $u_j,\ Q_j \in \ell(X)$
with $u_j>0$ and $Q_j$ locally uniformly bounded in $X$,
such that $-\Delta u_j + Q_j u_j = 0$. Assume that
$$
Q_\infty(x)=\sup \lim_{j\to\infty} Q_j(x).
$$
Then for any fixed point $x_0\in X$, setting
$\check{u}_j(x)=u_j(x)/u_j(x_0)$, we can take a sub-convergent
sequence, still denoted by ($\check{u}_j$) with its limit
$\check{u}_\infty>0$ such that
\[\label{limit}
-\Delta \check{u}+Q_\infty(x)\check{u}_\infty=0, \ \  in \ \ X.
\]
\end{theorem}
\begin{pf}
We may assume that $u_j(x_0)=1$ after the normalization. Take the
exhaustion $\Omega_j$ of $X$ such that $X=\bigcup_j\Omega_j$,
$\Omega_j\subset\subset \Omega_{j+1}$ and $x_0\in \Omega_j$ for each
$j\geq 1$. Then take any finite set $S$. Note that $\inf_S
u_j(x)\leq 1$. By using (\ref{bound}), we have a uniform constant
$C_S>0$ such that
$$
\sup_S u_j(x)\leq C_S.
$$
Taking the diagonal convergent subsequence we may assume that
$$
u_j(x)\to u_\infty(x), \ \  in \ \  S.
$$
Then it is clear that $u_\infty(x)\geq 0$, $u_\infty(x_0)=1$, and
$u_\infty$ satisfies (\ref{limit}). By using the maximum principle
\cite{MW}, we know that $u_\infty>0$ in $X$.

\end{pf}

Recall that for any finite subgraph $D$, we can define
$\lambda_1(D,-\Delta+Q)$ by
$$
\inf\{\frac{\int_D(-\Delta u+Qu)u}{\int_D u^2}; u\not=0; u|_{\delta
D}=0\}.
$$
Take an exhaustion $\{D_j\}$ of $X$. Then
$\lambda_1(D_{j+1},-\Delta+Q)\leq \lambda_1(D_{j},-\Delta+Q)$ so
that we can define $\lambda_1(X,-\Delta+Q)=\lim \lambda_1(D_j,
-\Delta+Q)$. It is also easy to see that $\lambda_1(X)$ does not
depend on the choice of the exhaustion and then the definition is
well-done.

With the help of Kato's inequalities, we can study the existence
problem of the positive solution to the Schrodinger equation on $X$.

\begin{theorem}\label{existence}
Assume that $X$ is a locally finite connected graph with
 $\lambda_1(X,-\Delta+Q)>0$. Then there exists a positive function
 $u:X\to R_+$ such that
$$ (-\Delta+Q)u=0, \ \ \ in \ \ X.
$$
\end{theorem}
\begin{pf} We take the exhaustion $D_j$ of $X$ as above. Then we can
solve
$$
(-\Delta+Q)u=-Q, \ \ \ in \ \ D_j.
$$
with the Dirichlet boundary condition $u|_{\delta D_j}=0$ to get a
solution $v_j$ in $D_j$. Let $u_j=v_j+1$. Then $u_j$ satisfies
$$
(-\Delta+Q)u=0, \ \ \ in \ \ D_j.
$$
with positive Dirichlet boundary condition. By Kato's inequality we
know that
$$
(-\Delta+Q)u_j^-\leq 0.
$$
Here $u^-=\frac{1}{2}(|u|-u)$ the positive part of the function
$-u$. Then
$$
\lambda_1(D_j)\int_{D_j} u_j^{-2}\leq
\int_{D_j}[(-\Delta+Q)u_j^-]u_j^-\leq 0,
$$
which implies that $u_j^-=0$. Then by the maximum principle, $u_j>0$
in $D_j$. We normalize $u_j$ such that $u_j(x_0)=1$ as before. Using
the Harnack inequality again, we can extract a convergent
subsequence, which is still denoted by $(u_j)$ with its limit
$u_\infty$. Then $u_\infty\geq 0$, $u_\infty(x_0)=1$, and
$$
(-\Delta+Q)u_\infty=0, \ \ \ in \ \ X.
$$
Using the maximum principle, we know that $u_\infty>0$ in $X$.
\end{pf}

Using almost the same argument we can prove the result below.
\begin{theorem}\label{existences}
Assume that $X$ is a locally finite connected graph with
 $\lambda:=\lambda_1(X,-\Delta+Q)>0$. Then there exists a positive function
 $u:X\to R_+$ such that
$$ (-\Delta+Q)u=\lambda u, \ \ \ in \ \ X.
$$
\end{theorem}
We may omit the detailed proof of Theorem \ref{existences}.

We can generalize Theorem \ref{existence} into the following form.
\begin{theorem}\label{existence2}
Assume that $X$ is a locally finite connected graph with
 $\lambda_1(X,-\Delta+Q)>0$.  Then for any $f\geq 0$ in $L^1(X)$,there exists a positive function
 $u:X\to R_+$ such that
$$ (-\Delta+Q)u=f, \ \ \ in \ \ X.
$$
\end{theorem}
\begin{pf} We may assume that $f$ is nontrivial. We take the exhaustion $D_j$ of $X$ as above. Using the Dirichlet principle, we can
solve
$$
(-\Delta+Q)u=f, \ \ \ in \ \ D_j.
$$
with the Dirichlet boundary condition $u|_{\delta D_j}=0$ to get a
non-negative solution $u_j$ in $D_j$.  Then
$$
\lambda_1(X)\int_{D_j} u_j^{2}\leq
\int_{D_j}[(-\Delta+Q)u_j]u_j=\int_{D_j}fu_j.
$$
Using the Cauchy-Schwartz inequality we have $\int_{D_j}u_j^2 \leq
\lambda_1(X)^{-2}\int_X f^2$. Then we can extract a convergent
subsequence in $L^2_{loc}(X)$, which is still denoted by $(u_j)$
with its limit $u_\infty$. Then $u_\infty\geq 0$, and
$$
(-\Delta+Q)u_\infty=f, \ \ \ in \ \ X.
$$
Since $f$ is nontrivial, we know that $u_\infty$ is non-trivial.
Using the maximum principle, we know that $u_\infty>0$ in $X$.
\end{pf}

\section{Green functions}\label{sect3}

In this section we find conditions such that there is a global Green
function on $X$. Similar results in a complete non-compact
Riemannian manifold are well-known \cite{SY}.

We first recall the construction of the Green function on a
transient Markov process on an infinite connected graph $X$. Define
a transition probability $p(x,y):\ X^2 \to {\Bbb R}$ as follow
\[
p(x,y) = \left\{ \begin{array}{ll} \frac{\mu_{xy}}{d_x}, \quad & \mbox{ if } (x,y)\in {\mathcal E};\\
0, & \mbox{ otherwise.} \end{array}\right.
\]

We define the Markov operator $P:\ \ell(X) \longrightarrow \ell(X)$
as follow
\[
Pf(x) = \sum_{y \in X} p(x,y) f(y).
\]
Simply to understand this is treating $P$ as a infinite matrix $(
p(x,y))_{x,y\in X}$ and a function $f\in \ell(X)$ as a collum
vector. Then the Markov operator on a function $f$ is the product of
$P$ and $f$.

For each $n\ge 0$, we define $n$ step transition probability
inductively by
\[
p_n(x,y) = \sum_{z\in X} p_{n-1}(x,z) p(z,y) = \sum_{z\in X} p(x,z)
p_{n-1}(z,y),
\]
here $p_0(x,y) : = \delta_x(y)$ and $p_1(x,y) := p(x,y)$. It is
clear that we can write this to a matrix form $P^n =
(p_n(x,y))_{x,y\in X}$. We assume that the Markov process is {\it
transient}, i.e.,
\[ \label{markov}
\sum_{n=0}^\infty p_n(x,y) < \infty
\]
for all $x,y \in X$. Then we can define the Green operator
\[
G = \sum_{n=0}^\infty P^n
\]
(Recall that $P^0 = I$ the identity matrix). It follows that
\[
G = I + P G, \quad \mbox{ or} \quad  (I-P)G = I.
\]

The {\it Green function } $G(x,y)$ is the $(x,y)$-coordinate of the
matrix $G$, i.e., $G= \big( g(x,y) \big)_{x,y\in X}$. Recall that
{\it the Laplacian operator} $\Delta = P - I$, i.e., for a function
$f\in \ell(X),$
\[
\Delta f (x) = (P-I) f (x) =  P f (x) - f(x) = \sum_{(x,y)\in
{\mathcal E}} \frac{\mu_{xy}}{d_x} (f(y)- f(x)).
\]

Again we let $X$ be a locally finite, connected graph. Let $S$ be a
finite subset of $V$, the {\it subgraph} $G(S)$ generated by $S$ is
a graph, which consists of the vertex-set $S$ and all the edges $x
\sim y, \ x,y \in S$ as the edge set.  The boundary $\delta S$ of
the induced subgraph $G(S)$ consists of all vertices that are not in
$S$ but adjacent to some vertex in $S$. We assume that the subgraph
$G(S)$ is connected. Sometimes people may like to write
$$
\bar{S}=S\bigcup \delta S.
$$
By the above construction we can easily get the Green function $G_S$
on $S$ with $G_S(\cdot, y)|_{\delta S}=0$ for any $y\in S$. Denote
by $B(R)$ be the ball of radius $R>0$ with center at some fixed
point $x_0\in X$. Take any exhaustion $\Omega_j=B(j)$ of $X$ as in
the proof of Theorem \ref{compactness} and let $G_j:=G_{\Omega_j}$
be the Green function on $\Omega_j$ with Dirichlet boundary
condition. Note that $0\leq G_j\leq G_{j+1}$ on $X$ for any $j\geq
1$.

\begin{theorem}\label{Green}
 Assume that there exists a positive function $\phi:X\to R_+$ such
 that $\Delta \phi\leq 0$ and $\phi(x)\to 0$ as $d(x,x_0)\to
 \infty$. Then the limit $G(x,y)=\lim G_j(x,y)$ exists for any $x\not= y\in X$
 and $G(x,y)$ is the green function on $X$.
\end{theorem}

\begin{proof} We only need to show that $G_j(x,y)$ is uniformly
bounded for any $x\not y \in X$. That is to say, there exists an
uniform constant $C(d(x,y))$ such that
\[ \label{bbc}
G_j(x,y)\leq C(d(x,y)).
\]
Assume not. Then we have some constant $R_0>0$ such that
$$
m_j=\max \{G_j(x,y); x\in \delta B_y(R_0)\}\to \infty
$$
as $j\to\infty$. Here $B_y(R_0)=\{x\in X, d(x,y)<R_0\}$ is the ball
at center $y$. Note that $\sup_{\delta B_y(R_0)}v_j=\epsilon$. Let
$\epsilon>0$ be a small constant and let $j>R_0$.

  Define
  $$
v_j(x)=\epsilon m_j^{-1}G_j(x,y).
  $$
Then $0\leq v_j(x)\leq \epsilon$ for $x\in B(j)\backslash B_y(R_0)$.
We take $\epsilon>0$ small enough such that
$$
v_j(x)\leq \phi(x), \ \ on \ \delta B_y(R_0).
$$
Since $\phi(x)>0=G_j$ on $\delta B(j)$, by the maximum principle, we
know that
$$
v_j(x)\leq \phi(x), \ \ x\in B(j)\backslash B_y(R_0).
$$
Take some fixed ball $B(J)$ with $B_y(R_0)\subset B(J)$. Take any
$\mu>0$ small and choose $\epsilon>0$ small enough such that
$\epsilon m_j<\mu$. Then we use the maximum principle again to know
that
\[ \label{bb}
v_j(x)\leq \epsilon+ \mu G_J(x,y), \ \ x\in B_y(R_0)\backslash
\{y\}.
\]
Note that
$$
\Delta v_j=0, \ \ in \ B_y(j).
$$

Using the Harnack inequality (\ref{bound}) we can extract a
convergent subsequence $v_j$ with its limit $\tilde{v}$ satisfying
both $\Delta v=0$  and $v(x)\leq \phi(x)$ on $X\backslash \{y\}$. By
(\ref{bb}) we know that
$$
\tilde{v}(x)\leq \epsilon+ \mu G_J(x,y).
$$
Sending $\mu\to 0$ we get that $\tilde{v}(x)\leq
\epsilon=\sup_{\delta B_y(R_0)}\tilde{v}$ on $B_y(R_0)\backslash
\{y\}$. The latter equality implies that the harmonic function
$\tilde{v}$ obtains its maximum at the interior point in
$X\backslash \{y\}$. By the strong maximum principle we know that
$\tilde{v}=\epsilon$, which is a contradiction with the fact that
$\tilde{v}(x)\leq \phi(x)\to 0$ as $x\to\infty$. Therefore, the
bound (\ref{bbc}) is true. Hence the limit of $G_j$ exists and it is
the minimal Green function on $X$.
\end{proof}

Similar to the argument above we can obtain the following result.

\begin{theorem}\label{Green2}
 Assume that there exists a positive function $\phi:X\to R_+$ such
 that $\Delta \phi\leq 0$ and $\int_X \phi^p<\infty$ for some $p>1$.
 Then the limit $G(x,y)=\lim G_j(x,y)$ exists for any $x\not= y\in X$
 and $G(x,y)$ is the green function on $X$.
\end{theorem}
Since the argument is almost the same, we omit its detailed proof.

We can also prove the following result, which has its counterpart in
Riemannian geometry.

\begin{theorem}\label{Green3} Assume that $X$ is a locally finite connected graph with
 $\lambda_1(X)>0$. Then $X$ has a global Green function.
\end{theorem}

\begin{proof} The proof goes as in Theorem \ref{Green}. Recall that
we still have
\begin{equation} \label{uniform} 0\leq v_j(x)\leq
\epsilon, \ \ \ x\in B(j)\backslash B_y(R_0).
\end{equation}
 Let $j>1$ be large enough. Take a fixed
constant $k\leq j/2$ such that $B_y(R_0)\subset B(k)$.  We claim
that there is a uniform constant $C>1$ such that
\begin{equation}\label{uniform2} \int_{X\backslash B(2k)} \tilde{v}^2\leq C.
\end{equation}
This together with (\ref{uniform}) implies that
$$
\int_X \tilde{v}^2<\infty.
$$
Then, using Theorem \ref{Green2} we know that $\tilde{v}=\epsilon$
on $X$. Then $Vol(X)<\infty$. Define the test function $w_j$ such
that $w_j=1$ on $B(j)$ and $w_j=0$ outside $B(j)$. Then $|\nabla
w_j|\leq 1$. By the definition of $\lambda_1(X)$, we know that
$$
\lambda_1(X)\leq \frac{\int_X |\nabla w_j|^2}{\int_X w_j^2}\leq
\frac{vol(B(j+1)\backslash B(j)}{vol(B(j))}\to 0,
$$
which is impossible.

To conclude the proof of Theorem \ref{Green3}, we now need only to
show the Claim (\ref{uniform2}). Define the test function $\eta$
such that $\eta=0$ on $B(k)$ and $\eta=1$ on $X\backslash B(2k)$,
and
$$
\eta(x)=\frac{d(x,x_0)-k}{k}, \ \ \ x\in B(2k)\backslash B(k).
$$

Then $|\nabla \eta|\leq C/k$, $|\Delta \eta|\leq c/k$, and $|\Delta
\eta^2|\leq 4/k$ for some uniform constant $C\geq 4$. Recall that
$$
\eta v_j\Delta (\eta v_j)=\eta v_j(\nabla \eta,\nabla
v_j)+v_j^2\eta\Delta\eta.
$$
The difficult term is $\eta v_j(\nabla \eta,\nabla v_j)$. Note that
by Lemma \ref{lemma2},
$$
\int_X \eta^2 |\nabla v_j|^2\leq \int_X \eta^2 \Delta v_j^2=\int_X
\Delta\eta^2  v_j^2,
$$
which is bounded by $\frac{C^2}{k^2}\int_{B(2k)\backslash
B(k)}v_j^2$. Then
$$
|\int_X\eta v_j(\nabla \eta,\nabla v_j)|\leq [\int_X\eta^2 |\nabla
v_j|^2\int_X v_j^2|\nabla \eta|^2]^{1/2}\leq
\frac{C^2}{k^2}\int_{B(2k)\backslash B(k)}v_j^2,
$$
which is again bounded by $\frac{C^2\epsilon^2}{k^2}$.

Hence,
$$
\lambda_1(X)\int_X(\eta v_j)^2\leq -\int_X (\Delta (\eta v_j),\eta
v_j)\leq \frac{2C^2\epsilon^2}{k^2}.
$$
Sending $j\to\infty$, this proves the Claim. Therefore, the proof of
Theorem \ref{Green3} is complete.
\end{proof}

We now use the Green function to study the principal eigenvalue on
$X$.
\begin{theorem}\label{eigen}
Assume that $X$ is a locally finite connected graph with the global
Green function $G(x,y)$.
 Assume that there exists a positive constant $A$ such that
 $\sup_X\int_X G(x,y)dy\leq A$.
 Then $\lambda_1(X)\geq A^{-1}$.
\end{theorem}

\begin{proof} Take any exhaustion $\{\Omega_j\}$ of $X$ as in
the proof of Theorem \ref{compactness} and let $G_j:=G_{\Omega_j}$
be the Green function on $\Omega_j$ with Dirichlet boundary
condition. According to the construction of the Green function, we
know that $G(x,y)=\lim G_j(x,y)$. By the assumption, we know that$$
\sup_X\int_{\Omega_j} G_j(x,y)dy\leq A.
$$
Recall that $\lambda_1(X)=\lim \lambda_1(\Omega_j)$ and there is a
positive function $u_j:\Omega_j\to R_+$ such that
$$
-\Delta u_j=\lambda_1(\Omega_j)u_j, \ \ \ in \ \ \Omega_j
$$
with the Dirichlet boundary condition $u_j|_{\delta \Omega_j}=0$.
Since
$$
u_j(x)=\lambda_1(\Omega_j)\int_{\Omega_j} G_j(x,y)u_j(y)dy,
$$
we have
$$
u_j(x)\leq \lambda_1(\Omega_j)\sup_{y\in \Omega_j}u_j(y)A.
$$
Let $B_j= \sup_{x\in \Omega_j}u_j(x)$. Then $B_j>0$ and
$$
B_j\leq \lambda_1(\Omega_j)B_jA.
$$
Hence, $\lambda_1(\Omega_j)A\geq 1$, which implies that
$\lambda_1(X)\geq A^{-1}$. This completes the proof of the result.
\end{proof}

\end{document}